\documentclass[10pt, a4paper, leqno]{amsart}
\usepackage{amsfonts, amssymb, amsmath}
\usepackage{amsthm}
\usepackage{enumerate}

\newtheorem{theorem}{Theorem}[section]

\newtheorem{lemma}[theorem]{Lemma}
\newtheorem{corollary}[theorem]{Corollary}
\theoremstyle{definition}
\newtheorem{remark}[theorem]{Remark}

\numberwithin{equation}{section}

\newcommand{\N}{{\mathbb N}}
\newcommand{\R}{{\mathbb R}}
\newcommand{\eps}{{\varepsilon}}

\newcommand{\U}{{\overline{U}}}
\renewcommand{\u}{{\underline{U}}}

\newcommand{\beq}{\begin{equation}}
\newcommand{\eeq}{\end{equation}}
\newcommand{\beqa}{\begin{eqnarray}}
\newcommand{\eeqa}{\end{eqnarray}}
\newcommand{\beqn}{\begin{eqnarray*}}
\newcommand{\eeqn}{\end{eqnarray*}}

\begin{document}

\title
[Supercritical elliptic equations with Hardy potential]
{Existence, stability and oscillation properties of slow decay
positive solutions of supercritical elliptic equations with Hardy potential}

\author{Vitaly Moroz}
\address{Swansea University\\ Department of Mathematics\\ Singleton Park\\
Swansea\\ SA2~8PP\\ Wales, United Kingdom}	
\email{V.Moroz@swansea.ac.uk}

\author{Jean Van Schaftingen}
\address{Universit\'e catholique de Louvain\\
Institut de Recherche en Math\'ematique et Physique (IRMP)\\
Chemin du Cyclotron 2 bte L7.01.01\\
1348 Louvain-la-Neuve\\
Belgium}
\email{Jean.VanSchaftingen@uclouvain.be}


\date{\today}

\begin{abstract}
We prove the existence of a family of slow decay positive solutions of
a supercritical elliptic equation with Hardy potential in $\R^N$
and study stability and oscillation properties of these solutions.
We also establish the existence of a continuum of stable slow
decay positive solutions for the relevant exterior Dirichlet problem.
\end{abstract}

\maketitle

\section{Introduction.}
Our starting point is a superlinear elliptic problem in the entire space
\begin{equation}\label{P}
-\Delta u = u^p,\quad u>0\quad\text{in $\R^N$},
\end{equation}
where $p>1$ and $N\ge 3$. By $p_S:=\frac{N+2}{N-2}$ in what follows we denote
the {\em critical Sobolev exponent}.
It is well--known that for $p<p_S$ problem \eqref{P} has no positive solutions.
For finite energy solutions this is an easy consequence of Pohozaev's identity.
For positive solutions without decay assumptions at infinity
this is a deep result of Gidas and Spruck \cite{GS}.
For $p=p_S$ all positive solutions of \eqref{P} are given up to translations by
a one-parameter family
$$W_{\lambda}(|x|):=\lambda^\frac{N-2}{2}W_1\big(\lambda|x|\big)\qquad(\lambda>0),$$
where $W_1(x):=\big(1+(N(N-2))^{-1}|x|^2\big)^{-\frac{N-2}{2}}$ is a rescaled minimizer
of the Sobolev inequality.

For $p>p_S$ the structure of the solution set of \eqref{P} is more complex.
First we note that for all $p>\frac{N}{N-2}$ problem \eqref{P} possesses an explicit {\em singular} radial positive solution
$$U_\infty(x):=C_p|x|^{-\frac{2}{p-1}},\qquad C_p:=\left(\frac{2}{p-1}\Big(N-2-\frac{2}{p-1}\Big)\right)^\frac{1}{p-1}.$$
Observe that if $p>p_S$ then $U_\infty\in H^1_{loc}(\R^N)$ and hence $U_\infty$ is a weak solution of \eqref{P} in the entire $\R^N$, despite the singularity at the origin.
However $U_\infty$ is an infinite energy solution because of its slow decay at infinity for $p>p_S$.

The set of all radially symmetric solutions of \eqref{P} can be analyzed
through phase plane analysis after applying Fowler's transformation, cf.~\cite[p.~50-55]{QS}.
In particular, if $p>p_{S}$ then \eqref{P} admits a radial positive solution $U_1(|x|)$
such that $U_1(0)=1$. It is known that $U_1(|x|)$ is monotone decreasing and
\[
  \lim_{|x|\to\infty}\frac{U_1(|x|)}{|x|^{-\frac{2}{p-1}}}=C_p,
\]
however $U_1$ has no explicit representation in terms of elementary functions.
Taking into account the scaling invariance one concludes that rescalings of $U_1$
are also solutions of \eqref{P}, so that \eqref{P} possess
a one-parameter continuum of radial positive solutions
\beq\label{family}
U_{\lambda}(|x|)=\lambda^{\frac{2}{p-1}}U_1(\lambda|x|)\qquad(\lambda>0).
\eeq
One can show that the singular solution $U_\infty$ is the limit of the family $(U_\lambda)$,
in the sense that for any $x\neq 0$ holds
$$\lim_{\lambda\to\infty}U_{\lambda}(|x|)=U_\infty(|x|).$$
In addition, it is known that given $0<\lambda_1<\lambda_2\le\infty$,
solutions $U_{\lambda_1}(r)$ and $U_{\lambda_2}(r)$ in the range $p_S<p<p_{JL}$
intersect each other infinitely many times as $r\to\infty$ , while
for $p\ge p_{JL}$ the solutions are strictly ordered, that is
$U_{\lambda_1}(r)<U_{\lambda_2}(r)$ for all $r\ge 0$.
Here
$$p_{JL}:=\left\{
\begin{aligned}
&\frac{N - 2\sqrt{N - 1}}{N - 4 - 2\sqrt{N - 1}},& &\text{if $N>10$},\\
&\infty& & \text{if $N\le 10$},\\
\end{aligned}
\right.$$
is the {\em Joseph--Lundgren stability exponent}, introduced in \cite{JL}.
The exponent $p_{JL}$ controls various oscillation and stability properties
of solutions $U_\lambda$, which are particularly important in the study
of the time--dependent parabolic version of \eqref{P},
see \cite{Ni,Wang} or \cite[p.~50-55]{QS} for a discussion.
\medskip

We are interested in a perturbation of \eqref{P} by the Hardy inverse square potential,
that is the equation
\begin{equation}\label{P-H-intro}
-\Delta u +\frac{\mu}{|x|^2}u= u^p,\quad u>0\quad\text{in $\R^N\setminus\{0\}$,}
\end{equation}
where $\mu > -C_H$ and $C_H:=\frac{(N-2)^2}{4}$ is the {\em Hardy critical constant},
i.e. the optimal constant in the Hardy inequality
\beq\label{Hardy-ineq}
\int_{\R^N}|\nabla\varphi|^2\,dx\ge C_H\int_{\R^N}\frac{|\varphi|^2}{|x|^2}\,dx\qquad\forall \varphi\in C^\infty_c(\R^N).
\eeq
Hardy potential provides an important example of a long range potential,
that is a potential which modifies asymptotic decay rate of solutions
at infinity and their behavior at the origin, see e.g. \cite{Bidaut-Veron,Guerch-Veron}.

For $p\neq p_S$ a Pohozaev--type identity shows that similarly to \eqref{P}, equation \eqref{P-H-intro} has no finite energy solutions \cite{Terracini}. For $p=p_S$ equation\eqref{P-H-intro} admits an explicit one-parameter family of finite energy radial solutions, cf. \cite{Bidaut-Veron,Terracini}. However, the structure of positive solutions of \eqref{P-H-intro} in the {\em critical regime} $p=p_S$ is not fully understood. It is known that for large values of $\mu>0$ equation \eqref{P-H-intro} admits nonradial solutions which are distinct modulo rescalings from the radial solutions \cite{Bidaut-Veron,Terracini}.
See \cite{YanYanLi} for recent results and discussion of open questions in this direction.

In the present work we consider equation \eqref{P-H-intro} in the {\em supercritical regime} $p>p_S$.
In the next section we setup the problem and discuss basic properties of the explicit singular solution
similar to $U_\infty$. In Section \ref{RN}, for optimal ranges of $p$ and $\mu$ we establish the existence of a one-parameter family $(U_\lambda)_{\lambda>0}$ of infinite energy solutions of \eqref{P-H-intro}, which coincides with \eqref{family} when $\mu=0$. We also discuss stability properties of these solutions. The presence of the Hardy potential produces a range of new critical exponents related to stability which do not have immediate analogues in the unperturbed case of equation \eqref{P}. Finally in Section \ref{S-ext}, we discuss equation \eqref{P-H-intro} in exterior domains. We justify optimality of critical exponents introduced in previous sections. Further, under some assumptions on $p$ and $\mu$ we prove the existence of a continuum of infinite energy solutions of \eqref{P-H-intro}, which in some sense could be considered as a perturbation of the original family of solutions on $\R^N$ but goes beyond spherically symmetric or scaling invariant setting. This partially extends some of the recent results in \cite{DelPino-CV}.

\section{Equations with Hardy potential.}
We study the equation
\begin{equation}\label{P-H}
-\Delta u +\frac{\nu^2-\nu_\ast^2}{|x|^2}u= u^p,\quad u>0\quad\text{in $\R^N\setminus K$,}
\end{equation}
where $ K=\{0\}$, or $\{0\}\in K$ and $K$ is a connected compact set with the smooth boundary $\partial K$,
$p>1$, $N\ge 3$, $\nu>0$ and $\nu_\ast:=\frac{N-2}{2}$,
so that $\nu_\ast^2$ is the Hardy critical constant in \eqref{Hardy-ineq}.
By a solution of \eqref{P-H} we understand a classical solution $u\in C^2(\R^N\setminus K)$,
with no apriori assumption on the decay of $u(x)$ at infinity.
We say $u$ is a weak solution of \eqref{P-H} in $\R^N$ if $u\in H^1_{loc}(\R^N)$ and
$$
\int\nabla u\cdot\nabla\varphi\,dx+\big(\nu^2-\nu_\ast^2\big)\int\frac{u\varphi}{|x|^2}\,dx
=\int u^p\varphi\,dx\qquad\forall \varphi\in C^\infty_c(\R^N).
$$
Note that for $\nu<\nu_\ast$ solutions of \eqref{P-H} must have a singularity at the origin (see Lemma \ref{Phragmen} below)
however this singularity might be compatible with the concept of a weak solution in $\R^N$.

We say a solution $u$ of \eqref{P-H} in $\R^N\setminus K$ has {\em finite energy} if $u\in D^1(\R^N\setminus K)$,
the completion of $C^\infty_c(\R^N\setminus K)$ with respect to the norm $\|\nabla\varphi\|_{L^2}$.
We say a solution $u$ of \eqref{P-H} is {\em stable} in $\R^N\setminus K$ if
the formal second variation at $u$ of the energy which corresponds to \eqref{P-H} is nonnegative definite,
that is
\beq\label{stable}
\int|\nabla\varphi|^2\,dx+\big(\nu^2-\nu_\ast^2\big)\int\frac{\varphi^2}{|x|^2}\,dx
-p\int u^{p-1}\varphi^2\,dx\ge 0\qquad\forall \varphi\in C^\infty_c(\R^N\setminus K).
\eeq
A solution $u>0$ of \eqref{P-H} is called {\em semi-stable} in $\R^N\setminus K$ if it is stable
in $\R^N\setminus B_R$, for some $R>0$.
A solution $u>0$ of \eqref{P-H} is called {\em unstable} if it is not semistable.
Note that these definitions do not require $u$ to be a finite energy solution.

\subsection{Explicit radial solution.}
For $\nu>0$ and $p>p_*:=1+\frac{2}{\nu_*+\nu}$, set
$$
U_\infty(x):=C_{p,\nu}|x|^{-\frac{2}{p-1}},\qquad C_{p,\nu}^{p-1}:=\nu^2-\left(\nu_*-\frac{2}{p-1}\right)^2,
$$
and introduce the critical exponent
$$p^*:=\left\{
\begin{array}{cl}
1+\frac{2}{\nu_*-\nu},&\text{if $\nu<\nu_*$},\\
\infty&\text{if $\nu\ge\nu_*$},\\
\end{array}
\right.$$
Clearly $p^*>p_S$.
A direct computation shows that $U_\infty$ is a positive solution of \eqref{P-H} for all $p_* < p < p^*$, while for $p \not\in [p_*, p^*]$ the coefficient $C_{p,\nu}$ becomes negative.
Note that $U_\infty\in H^1_{loc}(\R^N)$ for $p>p_S$,
that is $U_\infty$ is a weak solution of \eqref{P-H} in $\R^N$.
However $U_\infty$ is an infinite energy solution because of its slow decay at infinity.

The importance of the solution $U_\infty$ is due to the fact that it will be used as
an elementary building block for constructing further solutions of \eqref{P-H}.
In order to do this it is essential to understand stability properties of $U_\infty$.

\begin{lemma}
Let $p\in(p_*,p^*)$ and $\nu>0$.
The solution $U_\infty$ is stable if and only if
\beq\label{alg}
pC_{p,\nu}^{p-1}\le\nu^2,
\eeq
while if \eqref{alg} fails then $U_\infty$ is unstable.
\end{lemma}

\proof
The formal second variation of the energy which corresponds to \eqref{P-H} at $U_\infty$ is
given by
$$
\int|\nabla\varphi|^2\,dx+\big(\nu^2-\nu_\ast^2-pC_{p,\nu}^{p-1}\big)\int\frac{|\varphi|^2}{|x|^2}\,dx.
$$
Thus the assertion follows directly from the fact that $\nu_*^2$ is the optimal constant in
the Hardy inequality \eqref{Hardy-ineq}.

Taking into account the scaling invariance of Hardy's inequality
we also  conclude that if \eqref{alg} fails then $U_\infty$ must be unstable.
\qed

The inequality \eqref{alg} amounts to a third degree algebraic expression
for which closed form solutions could be obtained using Cardano's formulae,
however the explicit expressions for solutions are tedious.
Below we present a qualitative analysis of \eqref{alg}.
Set $s:=-\frac{2}{p-1}$, so that \eqref{alg} transforms into
$$\frac{(s+\nu_*)^2(s-2)+2\nu^2}{|s|}\le 0\qquad(-\nu_*-\nu<s<\min\{-\nu_*+\nu,0\}).$$
Define
$$\theta(s):=(s+\nu_*)^2(s-2).$$
Then solving \eqref{alg} for $p_*<p<p^*$ is equivalent to classifying the roots of the equation
\beq\label{roots}
\theta(s)= -2\nu^2\qquad(-\nu_*-\nu<s<\min\{-\nu_*+\nu,0\}),
\eeq
and solving the inequality
\beq\label{theta}
\theta(s)\le -2\nu^2\qquad(-\nu_*-\nu<s<\min\{-\nu_*+\nu,0\}).
\eeq
Note that $\theta(0)=-2\nu_*^2$ and that $\theta$ has two critical points:
a local maximum at $s_{max}:=-\nu_*$ with $\theta(s_{max})=0$ and
a local minimum at $s_{min}:=-\frac{\nu_*-4}{3}$ with
$\theta(s_{min})=-\frac{4}{27}(2+\nu_*)^3$. Denote
$$\bar\nu:=\sqrt{\frac{2}{27}(2+\nu_*)^3}=\sqrt{2\left(\frac{N+2}{6}\right)^3}.$$
Clearly for every $\nu>0$ equation \eqref{roots} has exactly one root $\sigma_\#$ in the interval $(-\nu_*-\nu,-\nu_*)$.
To analyze the roots of \eqref{theta} in the interval $(-\nu_*,\min\{-\nu_*+\nu,0\})$ we distinguish the cases $s_{min}<0$
and $s_{min}\ge 0$.
\footnote{Note that if we write $\bar\mu=\bar\nu^2-\nu_*^2$ as in \eqref{P-H-intro} then $\bar\mu=\frac{1}{108}(N-10)^2(N-1)$.}
\smallskip

In the case $s_{min}\ge 0$ (that is $3\le N\le 10$):

$\quad(i)$ if $\nu\ge\nu_*$ then \eqref{roots} has no roots in $(-\nu_*,0)$ and \eqref{theta} holds
$\forall s\in(-\nu_*-\nu,\sigma_\#]$,

$\quad(ii)$ if $0<\nu<\nu_*$ then \eqref{roots} has exactly one root $\sigma_-\in(-\nu_*,-\nu_*+\nu)$
and \eqref{theta} holds $\forall s\in(-\nu_*-\nu,\sigma_\#]\cup[\sigma_-,-\nu_*+\nu)$.
\smallskip

In the case $s_{min}< 0$ (that is $N> 10$):

$\quad(i)$ if $\nu>\bar\nu$ then \eqref{roots} has no roots in $(-\nu_*,0)$ so that \eqref{theta} holds
$\forall s\in(-\nu_*-\nu,\sigma_\#]$,

$\quad(ii)$ if $\nu_*<\nu\le\bar\nu$ then \eqref{roots} has exactly 2 roots $\sigma_-$ and $\sigma_+$
in $(-\nu_*,0)$ and $-\nu_*<\sigma_-\le s_{min}\le\sigma_+<0$ so that \eqref{theta} holds
$\forall s\in(-\nu_*-\nu,\sigma_\#]\cup[\sigma_-,\sigma_+]$,

$\quad(iii)$ if $0<\nu\le\nu_*$ then \eqref{roots} has exactly 1 root $\sigma_-$ in $(-\nu_*,0)$ and
$\sigma_-\in(-\nu_*,s_{min})$ so that \eqref{theta} holds $\forall s\in(-\nu_*-\nu,\sigma_\#]\cup[\sigma_-,0)$.
\smallskip

In what follows we denote
$$p_\#=1-\frac{2}{\sigma_\#},\qquad p_-:=1-\frac{2}{\sigma_-},\qquad p_+:=1-\frac{2}{\sigma_+},$$
and note that
$$1<p_*<p_\#<p_S<p_-\le p_+<p^*,$$
for all values of $N\ge 3$ and $\nu>0$ when all the exponents are well defined.
Then the above analysis leads to the following equivalent to \eqref{alg}
characterization of the stability properties of the solution $U_\infty$ in terms of the original parameters $p$ and $\nu$.

\begin{lemma}
Let $p\in(p_*,p^*)$ and $\nu>0$.
\begin{itemize}
\item[(a)]
If $\nu_*<\nu\le\bar\nu$ and $N\ge 11$ then the solution $U_\infty$
is stable for $p\in(p_*,p_\#]\cup[p_-,p_+]$ and unstable for $p\in(p_\#, p_-)\cup(p_+,p^*)$.

\item[(b)]
If $0<\nu<\nu_*$ and $N\ge 3$ or $\nu=\nu_*$ and $N\ge 11$ then
the solution $U_\infty$ is stable for $p\in(p_*,p_\#]\cup[p_-,p^*)$ and unstable for $p\in(p_\#, p_-)$.

\item[(c)]
If $\nu\ge\nu_*$ and $3\le N\le 10$ or $\nu\ge\bar\nu$ and $N\ge 11$
the solution $U_\infty$ is stable for $p\in(p_*,p_\#]$ and unstable for $p\in(p_\#,\infty)$.
\end{itemize}
\end{lemma}

\begin{remark}
In the pure Laplacian case $\nu=\nu_*$ one calculates the explicit values
$$p_\#=\frac{N + 2\sqrt{N - 1}}{N - 4 + 2\sqrt{N - 1}},\qquad p_-=\frac{N - 2\sqrt{N - 1}}{N - 4 - 2\sqrt{N - 1}},$$
here $p_-$ is defined only for $N\ge 11$.
Thus for the Laplacian the exponent $p_-$ coincides with the Joseph--Lundgren stability exponent,
see \cite{JL} or \cite[p.50]{QS}; while the exponent $p_\#$ is known to appear
in the context of local singularities of solution of equations \eqref{P}, cf. \cite[Lemma 5]{Pacard}.
\end{remark}

\begin{remark}
If $N>10$ and $\nu=\bar\nu$ then $p_-=p_+=\frac{N+2}{N-10}$ is the only supercritical value
of $p$ where $U_\infty$ is stable.
\end{remark}

\subsection{Slow and fast decay solutions.}
Clearly, a solution $u$ of \eqref{P-H} is a positive superharmonic of the linear Hardy operator,
that is $u$ satisfy the linear inequation
\begin{equation}\label{L-H}
-\Delta u +\frac{\nu^2-\nu_*^2}{|x|^2}u\ge 0\quad\text{in $\R^N\setminus K$.}
\end{equation}
As a consequence, solutions of \eqref{P-H} with $\nu^2<\nu_*^2$ are always singular at the origin while for
$\nu^2>\nu_*^2$ solutions might vanish at the origin. More precisely, the following local decay properties for
positive superharmonics of Hardy's operator hold, cf. \cite{LLM}.
\begin{lemma}\label{Phragmen}
If $u>0$ satisfy \eqref{L-H} in a neighborhood of the origin then
\beq\label{Phragmen-0}
\liminf_{|x|\to 0}\frac{u(x)}{|x|^{-\nu_\ast+\nu}}>0,\qquad \liminf_{|x|\to 0}\frac{u(x)}{|x|^{-\nu_\ast-\nu}}<\infty.
\eeq
If $u>0$ satisfy \eqref{L-H} in an exterior domain then
\beq\label{Phragmen-infty}
\liminf_{|x|\to\infty}\frac{u(x)}{|x|^{-\nu_\ast-\nu}}>0,\qquad \liminf_{|x|\to 0}\frac{u(x)}{|x|^{-\nu_\ast+\nu}}<\infty.
\eeq
\end{lemma}
Bidaut--V\'eron and V\'eron \cite[Theorem 3.3]{Bidaut-Veron} proved that the structure of the solution set of \eqref{P-H}
in exterior domains which decay at infinity no slower then $U_\infty$ is essentially determined
by the solutions of the following equation
\beq\label{P-SN}
-\Delta_{S^{N-1}}\omega+C_{p,\nu}^{p-1}\omega=\omega^p,\quad\omega>0\quad\text{in $S^{N-1}$.}
\eeq
on the sphere $S^{N-1}$.

\begin{lemma}{\cite[Theorem 3.3]{Bidaut-Veron}}\label{lemma-Veron}
Let $p\neq p_S$.
If $u>0$ satisfy \eqref{P-H} in $\R^N\setminus K$ and
\beq\label{apriori}
\limsup_{|x|\to\infty}\frac{u(x)}{|x|^{-\frac{2}{p-1}}}<\infty,
\eeq
then either
\beq\label{fast}
\lim_{|x|\to\infty}\frac{u(x)}{|x|^{-\nu_\ast-\nu}}=c\qquad\qquad\text{(fast decay)},
\eeq
or there exists a positive solution $\omega(\cdot)$ of \eqref{P-SN} such that
\beq\label{slow}
\lim_{|x|\to\infty}\frac{u(|x|,\cdot)}{|x|^{-\frac{2}{p-1}}}=\omega(\cdot)\qquad\quad\text{(slow decay)}
\eeq
in the $C^k(S^{N-1})$ topology, for any $k\in\N$.
\end{lemma}

\begin{remark}
Clearly, $C_{p,\nu}$ is a constant solution of \eqref{P-SN}. For $1<p<\frac{N+1}{N-3}$
it is known (see \cite{GS} or \cite[Corollary 6.1]{Bidaut-Veron}) that $C_{p,\nu}$ is the only solution of \eqref{P-SN} provided that
\beq\label{alg-SN}
(p-1)C_{p,\nu}^{p-1}\le N-1,
\eeq
while if \eqref{alg-SN} fails then problem \eqref{P-SN} admits nonconstant
solutions, see \cite[Corollary 6.1]{Bidaut-Veron}, \cite[Theorem 0.5]{Terracini} and \cite[Theorem 1.3]{YanYanLi}.
Similar result holds for some values $p>\frac{N+1}{N-3}$, see \cite{Brezis-Li}.
The complete structure of solution set of \eqref{P-SN} is not yet fully understood,
see \cite{YanYanLi,Veron-Ponce} for some recent results in this direction.
\end{remark}

\begin{remark}
If $\nu>0$ and $p<p_S$ then \eqref{apriori} always holds, see \cite[Remark 3.2]{Bidaut-Veron}.
\end{remark}

We will classify positive solutions of \eqref{P-H} into {\em fast} and {\em slow decay solutions}
according to alternatives \eqref{fast} and \eqref{slow}. Note that for $p>p_S$ slow decay solutions are always infinite energy solutions, because of the slow decay rate \eqref{slow} at infinity.

\section{Radial slow decay solutions in $\R^N$.}\label{RN}
Radial positive solutions $u(|x|)>0$ of \eqref{P-H} in $\R^N\setminus\{0\}$
correspond to the positive solutions $U(r)=u(r)$ of the initial value problem
\begin{equation}\label{P-ODE}
-U''-\frac{N-1}{r}U' +\frac{\nu-\nu_*}{r^2}U=U^p\qquad (r>0),
\end{equation}
which can be studied through the phase plane analysis.

The existence of a family of regular at the origin slow--decay solutions of \eqref{P-ODE} in the
Laplacian case $\nu=\nu_*$ is well--known and goes back at least to \cite{JL}.

\begin{theorem}\label{t-U-lambda}
Let $p_S<p<p^*$. Then for any $\lambda>0$ equation \eqref{P-ODE} admits a unique
positive solution $U_\lambda\in C^2(0,\infty)$ such that
\beq\label{init-cond}
\lim_{r\to 0}\frac{U_\lambda(r)}{r^{-\nu_\ast+\nu}}=\lambda,\qquad
\lim_{r\to\infty}\frac{U_\lambda(r)}{r^{-\frac{2}{p-1}}}=C_{p,\nu}.
\eeq
Moreover,
\beq\label{rescaled}
U_\lambda(r)=\lambda^\frac{2}{p-1}U_1(\lambda r)\qquad\forall\lambda>0.
\eeq
Further, for $\lambda\in(0,\infty]$ the following properties hold:
\begin{itemize}

\item[(i)]
if $pC_{p,\nu}^{p-1}\le\nu^2$ then solutions $U_\lambda$ are stable and ordered
in the sense that $0<\lambda_1<\lambda_2\le\infty$ implies
$U_{\lambda_1}(r)<U_{\lambda_2}(r)$ for every $r\ge 0$ and in addition,
\beq\label{beta-plus}
\lim_{r\to\infty}\frac{U_{\lambda_2}(r)-U_{\lambda_1}(r)}{r^{-\nu_*}}>0;
\eeq
\smallskip

\item[(ii)]
if $pC_{p,\nu}^{p-1}>\nu^2$
then solutions $U_\lambda$ unstable and oscillate,
in the sense that $0<\lambda_1<\lambda_2\le\infty$
implies that $U_{\lambda_2}(r)-U_{\lambda_1}(r)$ changes sign in $(R,+\infty)$
for arbitrary $R>0$.

\end{itemize}

\end{theorem}

The proof of the theorem follows the exposition in \cite[pp.50-53]{QS}
with minor adjustments needed to accommodate $\nu\neq\nu_*$.
We present the sketch of the arguments for the readers convenience.

\begin{proof}[Proof of Theorem~\ref{t-U-lambda}]
Using the transformation
\beq\label{change}
w(t)=r^\frac{2}{p-1}U(r),\qquad t=\log(r),
\eeq
problem \eqref{P-ODE} becomes an autonomous second order differential equation
\beq\label{W-ODE}
w''+2\beta w'+w^p-\gamma w=0,\qquad t\in\R,
\eeq
where since \(p_S < p < p^*\),
$$\beta:=\nu_*-\frac{2}{p-1}>0,\quad \text{and} \quad  \gamma=C_{p,\nu}^{p-1}=\nu^2-\left(\nu_*-\frac{2}{p-1}\right)^2>0.$$

Set
$${\mathcal E}(w)={\mathcal E}(w,w'):=\frac{1}{2}|w'|^2-\frac{\gamma}{2}w^2+\frac{1}{p+1}w^{p+1}.$$
Then ${\mathcal E}$ is a Lyapunov function for \eqref{W-ODE} and
$$\frac{d}{dt}{\mathcal E}(w(t))=-2\beta(w'(t))^2\le 0.$$
Set $x:=w$ and $y:=w'$. Then \eqref{W-ODE} can be written as an autonomous first order system
$$
\begin{pmatrix}x' \\ y'\end{pmatrix}=\begin{pmatrix}y \\ -2\beta y+\gamma x-x^p\end{pmatrix}=:\Phi(x,y),
$$
which possesses two equilibria
$$(0,0)\quad\text{and}\quad(\gamma^\frac{1}{p-1},0)$$
in the half--space $\{(x,y):x\ge 0\}$. Denote
$$A_0:=\nabla\Phi(0,0)=\left(\begin{array}{cc}0 & 1 \\ \gamma & -2\beta\end{array}\right),
\qquad A_\ast:=\nabla\Phi(\gamma,0)=\left(\begin{array}{cc}0 & 1 \\ -(p-1)\gamma & -2\beta\end{array}\right).$$
The matrix $A_0$ has two real eigenvalues
$$\alpha_\pm:=-\beta\pm\sqrt{\beta^2+\gamma}=\frac{2}{p-1}-\nu_*\pm\nu,$$
so that $\alpha_-<0<\alpha_+$. The corresponding eigenvectors are $(1,\alpha_+)$ and $(1,\alpha_-)$,
that is $(0,0)$ is a saddle point of the vector field $\Phi$.
The matrix $A_\ast$ has two eigenvalues
$$\alpha^\ast_\pm:=-\beta\pm\sqrt{\beta^2-(p-1)\gamma},$$
the corresponding eigenvectors are $(1,\alpha^*_+)$ and $(1,\alpha^*_-)$.
Clearly $\mathrm{Re}(\alpha^\ast_\pm)<0$, so $(\gamma^\frac{1}{p-1},0)$ is always an attractor.
Note also that $\alpha^\ast_\pm$ is real if and only if $pC_{p,\nu}^{p-1}\le\nu^2$.

Using the Lyapunov function $\mathcal E$ one can show that the trajectory tangent
at the origin to the eigenvector  $(1,\alpha_+)$
is a heteroclinic orbit which connects the equilibria $(0,0)$ and $(\gamma^\frac{1}{p-1},0)$,
see \cite[p.52]{QS}. Moreover, since $(0, 0)$ is a hyperbolic saddle-point, the uniqueness of such hetereclinic orbit follows by standard arguments. The corresponding solution $w(t)$ exists for all $t\in\R$ and satisfies
\beq\label{init}
\lim_{t\to-\infty}w(t)=0,\qquad \lim_{t\to+\infty}w(t)=\gamma^\frac{1}{p-1}.
\eeq
Moreover, we can assume that $w(t)$ satisfies the normalization condition
\beq\label{normal}
\lim_{t\to -\infty}\frac{w(t)}{e^{\alpha_+ t}}=1.
\eeq
Since \eqref{W-ODE} is autonomous,
$w(t+\theta)$ is also a solution of \eqref{W-ODE}
that corresponds to the same heteroclinic orbit, for any $\theta\in\R$.
Given $\theta\in\R$, set $\lambda:=e^\theta$.
Then
$$U_\lambda(r):=r^{-\frac{2}{p-1}}w(\log(\lambda r))=\lambda^\frac{2}{p-1}U(\lambda r),$$
and $U_\lambda$ satisfies \eqref{init-cond} in view of \eqref{normal} and \eqref{init},
that is $U_\lambda$ is the required solution of \eqref{P-ODE}.
The uniqueness of $U_\lambda$ follows from the uniqueness of $w(t)$ since \eqref{change}
defines a one to one correspondence between solutions of \eqref{P-ODE} and \eqref{W-ODE}.

To understand oscillation and stability properties of $U_\lambda$ note
that the eigenvalues $\alpha^\ast_\pm$ are real iff
$$\beta^2\ge(p-1)\gamma,$$
which is equivalent to the stability condition \eqref{alg}.
Note that then
$$\alpha_-<\alpha_-^*\le-\nu_*\le\alpha_+^*<\alpha_+.$$
If the roots $\alpha^\ast_\pm$ are real then arguments similar to \cite[p.53]{QS}
show that the trajectory $w(t)$ is monotone increasing in $t$ for all $t\in\R$.
Hence the solutions $U_\lambda(r)$ are monotone increasing in $\lambda$.
In particular, $U_\lambda(r)<U_\infty(r)$ for any $\lambda>0$ and solutions
$U_\lambda$ are ordered.
Further, in view of \eqref{alg} the solution $U_\infty$ is stable.
Since $U_\lambda(r)<U_\infty(r)$, we obtain
$$pU_\lambda^{p-1}(|x|)\le pU_\infty^{p-1}(|x|)=p\gamma|x|^2\le\nu^2|x|^2.$$
By Hardy's inequality we conclude that
$$\int|\nabla\varphi|^2\,dx+\big(\nu^2-\nu_*^2)\int\frac{\varphi^2}{|x|^2}\,dx-p\int U_\lambda^{p-1}(|x|)\varphi^2\ge 0$$
for all $\varphi\in C^\infty_c(\R^N)$, that is $U_\lambda$ is a stable solution of \eqref{P-H}.
In addition, similarly to \cite[Remark 9.4]{QS}, we conclude that
$$\lim_{t\to\infty}\frac{w'(t)}{w(t)-\gamma^\frac{1}{p-1}}=\alpha^*_+\ge -\beta,$$
which after returning to the original variables and combined with \eqref{rescaled} implies \eqref{beta-plus}.
\smallskip

If $\alpha^\ast_\pm$ are complex then similarly to \cite[p.52]{QS} one can
see that the trajectory $(x(t),y(t))$ spirals infinitely many times around the attractor $(\gamma,0)$
which suggests that the solutions $U_\lambda$ oscillate in the sense of $(ii)$.
The detailed prove of oscillation and instability of $U_\lambda$ when $\alpha^\ast_\pm$ are complex is a particular
case of a more general Theorem \ref{t-osc} which will be proved in the next Section.
\end{proof}

\begin{remark}
In the subcritical case $p_*<p\le p_S$ equation \eqref{P-ODE} has no positive slow decay solution which satisfy \eqref{init-cond}.
Indeed, if $p=p_S$ then $\beta=0$, $\mathrm{Re}(\alpha^\ast_\pm)=0$ and the stationary point $(\gamma^\frac{1}{p-1},0)$ is a center. One can show that the trajectory tangent at the origin to the eigenvector $(1,\alpha_+)$ is a homoclinic orbit. This homoclinic corresponds to an explicit one parameter family of finite energy solutions of \eqref{P-ODE},
see \cite[pp.253-254]{Terracini}.
If $p_*<p<p_S$ then $\beta>0$, $\mathrm{Re}(\alpha^\ast_\pm)>0$ and the stationary point $(\gamma^\frac{1}{p-1},0)$ is repelling.
Hence a heteroclinic between  $(\gamma^\frac{1}{p-1},0)$ and $(0,0)$ originates at $(\gamma^\frac{1}{p-1},0)$
and converges to $(0,0)$ tangentially to the eigenvector $(1,\alpha_-)$. This heteroclinic corresponds to a positive solution of
\eqref{P-ODE} which decays at infinity as $O\big(|x|^{-\nu_\ast-\nu}\big)$ and has a singularity at the origin
of order $O\big(|x|^{-\frac{2}{p-1}}\big)$.
\end{remark}

\section{Slow decay solutions in exterior domains.}\label{S-ext}

First we justify that the value of the nonexistence exponent $p^*$ is sharp.
The result, which is first appeared in \cite[Remark 3.2]{Bidaut-Veron},
is an immediate consequence of Lemma \ref{lemma-Veron}.

\begin{theorem}
Let $p\ge p^*$. Then \eqref{P-H} has no slow decay solutions in $\R^N\setminus \bar B_R$,
for arbitrary $R>0$.
\end{theorem}

\proof
Simply note that for $p>p^*$ one has $C_{p,\nu}\le 0$ and hence
the equation \eqref{P-SN} on the sphere does not have any positive solution.
Then the conclusion follows from Lemma \ref{lemma-Veron}.
\qed

\begin{remark}
If $p>p^*$ then the slow decay rate is incompatible with the upper bound \eqref{Phragmen-infty}
of Lemma \ref{Phragmen}. This argument however does not apply when $p=p_*$.
\end{remark}

Next we justify sharpness of the stability and nonoscillation condition \eqref{alg}.
The result below extends oscillation statement of Theorem \ref{t-U-lambda} beyond radial setting.
See also \cite[Proposition 3.5]{Wang} for related results in the pure Laplacian case $\nu=\nu_*$.

\begin{theorem}\label{t-osc}
Let $p>p_S$, $\nu>0$ and $pC_{p,\nu}^{p-1}>\nu^2$.
Let $U_\ast>0$ be a subsolution of \eqref{P-H} such that
\beq\label{sub-ast}
\liminf_{|x|\to\infty}\frac{U_\ast(x)}{|x|^{-\frac{2}{p-1}}}\ge C_{p,\nu}.
\eeq
Then $U_\ast$ is unstable. Further, if  $u>0$ is a supersolution of \eqref{P} then either $u=U_\ast$,
or $\big(u-U_\ast\big)_-\neq 0$ in $\R^N\setminus \bar B_R$, for arbitrary large $R>0$.
\end{theorem}

\proof
From \eqref{sub-ast} we obtain
$$pU_\ast^{p-1}(x)\ge(\nu^2+\eps)|x|^{-2}\qquad(|x|>R_\eps),$$
for some $\eps>0$ and $R_\eps\ge R$.
Assume that $U_*$ is semistable, that is there exists $R>0$ such that \eqref{stable} holds in $\R^N\setminus \bar B_R$.
But then we arrive at
$$\int|\nabla\varphi|^2\,dx+\big(\nu^2-\nu_\ast^2\big)\int\frac{\varphi^2}{|x|^2}\,dx
\ge(\nu^2+\eps)\int\frac{\varphi^2}{|x|^2}\,dx\ge 0\quad\forall \varphi\in C^\infty_c(\R^N\setminus \bar B_{R_\eps}),$$
a contradiction to Hardy's inequality. We conclude that $U_*$ is unstable.

Further, set $h=u-U_\ast$ and assume that $h\ge 0$ in $\R^N\setminus \bar B_R$, for some $R>1$.
Then by convexity and \eqref{sub-ast} we obtain
\beqn
-\Delta h+\frac{\nu^2-\nu_*^2}{|x|^2}h&\ge& u^p-U_\ast^p=(U_\ast+h)^p-U_\ast^p\\
&\ge& pU_\ast^{p-1}h =\frac{pC_p^{p-1}}{|x|^2}h\ge\frac{\nu^2+\eps^2}{|x|^2}h\qquad(|x|>R_\eps).
\eeqn
It is well-known that such inequation has no positive solutions, cf. \cite[Corollary 3.2]{LLM}.
We conclude that either $h=0$, or $h$ changes sign in $\R^N\setminus \bar B_{R_\eps}$.
\qed

\begin{remark}
The above result does not exclude possibility that $u<U_*$ in an exterior domain.
The latter is however not possible in the case when both $U_*$ and $u$ are slow decay solutions.
In particular, since all the solutions $U_\lambda$ satisfy \eqref{sub-ast},
the above result includes the oscillation statement (ii) of Theorem \ref{t-U-lambda}.
\end{remark}

Next we show that if the stability assumption \eqref{alg} holds then
slow decay solutions of \eqref{P-H} in exterior domains are well ordered in a certain sense.
We consider the exterior boundary value problem for \eqref{P-H}
\begin{equation}\label{P-H-psi}
\left\{
\begin{array}{rcll}
-\Delta u +\frac{\nu^2-\nu_\ast^2}{|x|^2}u&=&u^p,\quad u>0 &\quad\text{in $\R^N\setminus K$,}\\
u&=&\psi&\quad\text{on $\partial K$,}
\end{array}
\right.
\end{equation}
here $K\ni\{0\}$ is a connected compact set with the smooth boundary $\partial K$,
and $\psi\in C(\partial K)$ is a nonnegative continuous function.

\begin{theorem}\label{t-order}
Let $p>p_S$, $\nu>0$ and $pC_{p,\nu}^{p-1}\le \nu^2$.
Let $U_\ast>0$ be a slow decay solution of \eqref{P-H} in $\R^N\setminus K$ such that for some $R>0$ holds
$$
U_*(x)\le U_\infty(x)\qquad(|x|>R).
$$
Given $\psi\in C(\partial K)$ such that
$$0\le \psi(x)\le {U_\ast}(x)\quad\text{on $\partial K$},$$
problem \eqref{P-H-psi} admits a slow decay solution $U_\ast^\psi$ such that
$$0<U_\ast^\psi\le U_\ast\quad\text{in $\R^N\setminus K$.}$$
Moreover,
\beq\label{beta-minus}
\lim_{|x|\to\infty}\frac{U_\ast(x)-U_\ast^\psi(x)}{|x|^{-\nu_*}}=0.
\eeq
\end{theorem}

\proof
We are going to construct a sub--solution $\u$ and a super--solution $\U$ such that
$$\text{$0\le \u\le \U\le U_\ast\;$ and $\;\u=\U=\psi$ on $\partial K$.}$$
Then the existence of a solution $U_\ast^\psi$ between $\u$ and $\U$ follows
via the classical sub and super--solution argument, cf. \cite[Theorem 38.1]{KZ}.
\medskip

{\sl Subsolution $\u$.}
Let $h_\psi>0$ be the minimal positive solution to the problem
\beq\label{lin}
-\Delta h +\frac{\nu^2-\nu_*^2}{|x|^2} h = pU_\ast^{p-1}h \quad\text{in $\R^N\setminus K$},\qquad h=U_\ast-\psi\quad\text{on $\partial K$}.
\eeq
The existence of such a solution is ensured by the Lax--Milgram theorem.
Indeed, by the assumptions
\beq\label{Hardy-super}
pU_\ast^{p-1}(x)\le pU_\infty^{p-1}(x)\le pC_{p,\nu}^{p-1}|x|^{-2}\le\nu^2|x|^{-2}.
\eeq
Hence the corresponding to \eqref{lin} quadratic form is coercive on the Sobolev
space $D^1_0(\R^N\setminus K)$. Moreover, from Lemma \ref{lemma-Veron} we conclude that
given a large $R>0$ there exists $m\in(0,\nu^2]$ such that
$$pU_\ast^{p-1}(x)\ge m|x|^{-2}\qquad(|x|>R).$$
Then a standard application of the comparison principle for Hardy operators (cf. Lemma \ref{Phragmen} and \cite[Lemma A.8]{LLM})
implies the two-sided bound
$$c|x|^{\alpha'_-}\le h_\psi\le C |x|^{\alpha^*_-}\qquad(|x|>R),$$
where $\alpha^*_-$ is the smallest root of
$$-(\alpha+\nu_*-\nu)(\alpha+\nu_*+\nu)=pC_{p,\nu}^{p-1}$$
and  $\alpha'_-$
is the smallest root of
$$-(\alpha+\nu_*-\nu)(\alpha+\nu_*+\nu)=m.$$
Note that $0<m\le pC_{p,\nu}^{p-1}\le\nu^2$, so both equations have real roots and
$$-\nu_*-\nu<\alpha'_-\le\alpha^*_-<-\nu_*<-\frac{2}{p-1}.$$
Set
$$\u:=U_\ast-h_\psi.$$
Then
$$\lim_{|x|\to\infty}\frac{\u(x)}{U_\ast(x)}=1,$$
and by convexity a direct computation shows
$$-\Delta \u = U_\ast^p-pU_\ast^{p-1}h_\psi\le (U_\ast-h_\psi)^p=\u^p\quad\text{in $\R^N\setminus K$},$$
that is $\u$ is the required sub--solution. In addition,
$$\lim_{|x|\to\infty}\frac{U_\ast(x)-\u(x)}{|x|^{-\nu_*}}=\lim_{|x|\to\infty}\frac{h_\psi(x)}{|x|^{-\nu_*}}=0,$$
which implies \eqref{beta-minus}.
\smallskip

{\sl Supersolution $\U$.}
Let $\eta_\psi>0$ be the minimal solution to the problem
$$-\Delta \eta +\frac{\nu^2-\nu_*^2}{|x|^2} h = U_\ast^{p-1}\eta \quad\text{in $\R^N\setminus K$},\qquad \eta=U_\ast-\psi\quad\text{on $\partial K$}.$$
Note that $U_\ast^{p-1}\le p U_\ast^{p-1}$.
Hence, solution $\eta_\psi$ exist simply because \eqref{Hardy-super} applies.
Moreover, a comparison argument similar to the ones above shows that
$$0<\eta_\psi<h_\psi\quad\text{in $\R^N\setminus K$}.$$
Define
$$\U:=U_\ast-\eta_\psi.$$
Then
$$\lim_{|x|\to\infty}\frac{\U(x)}{U_\ast(x)}=1,$$
and
$$-\Delta \U = U_\ast^{p-1}(U_\ast-\eta_\psi)\ge \big(U_\ast-\eta_\psi\big)^{p-1}(U_\ast-\eta_\psi)=\U^p
\quad\text{in $\R^N\setminus K$},$$
that is $\U$ is the required super--solution.
\qed

The next result shows that under suitable assumptions on the boundary data the exterior problem \eqref{P-H-psi}
admits a continuum of distinct slow decay positive solution, which in a certain sense could be interpreted
as a perturbation of the family of slow decay solutions $(U_\lambda)$ constructed in Theorem \ref{t-U-lambda}.

\begin{corollary}\label{cor-0}
Let $p>p_S$, $\nu>0$ and $pC_{p,\nu}^{p-1}\le \nu^2$.
Then for every $\psi\in C(\partial K)$ such that
\beq\label{psi-less}
0\le \psi(x)< {U_\infty}(x)\quad\text{on $\partial K$},
\eeq
problem \eqref{P-H-psi} admits a continuum of distinct positive slow decay solutions.
\end{corollary}

\proof
Consider the family of slow decay solution $(U_\lambda)_{\lambda>0}$, constructed in Theorem \ref{t-U-lambda}.
In view of \eqref{psi-less} there exists $\lambda_\psi>0$ such that for all $\lambda>\lambda_\psi$
$$0\le \psi(x)<{U_\lambda}(x)<U_\infty(x)\quad\text{on $\partial K$}.$$
Let $\lambda_1\in(\lambda_\psi,\infty]$.
In Theorem \ref{t-order}, choose $U_*:=U_{\lambda_1}$  and note that in view of \eqref{beta-minus} and \eqref{beta-plus}
the solution $U_{\lambda_1}^\psi$ given by Theorem \ref{t-order} is distinct with $U_\lambda$ for any $\lambda>\lambda_\psi$,
or with $U_{\lambda_2}^\psi$ for any other $\lambda_2>\lambda_\psi$. In such a way we have obtained
a family of distinct slow decay solutions $(U_\lambda^\psi)_{\lambda\in(\lambda_\psi,\infty]}$.
\qed

\begin{remark}
In particular, if  $p>p_S$ and $pC_{p,\nu}^{p-1}\le \nu^2$ then the problem
\beq\label{psi-0}
-\Delta u +\frac{\nu^2-\nu_\ast^2}{|x|^2}u = u^p \quad\text{in $\R^N\setminus K$},\qquad u=0\quad\text{on $\partial K$},
\eeq
admits a continuum of distinct positive slow decay solutions $(U^0_{\lambda})_{\lambda\in(0,\infty]}$.
This partially extends the result in \cite[Theorem 1]{DelPino-CV}), obtained in the pure Laplacian case
$\nu=\nu_*$. Note however that in \cite{DelPino-CV} the existence of a continuum of slow decay solutions
was proved for the whole range of exponents $p>p_S$, including the most challenging unstable regime $pC_{p,\nu_*}^{p-1}>\nu_*^2$.
The techniques in \cite{DelPino-CV} (see also a survey \cite{DelPino}) are based on linearization and perturbation arguments
combined with a sophisticated machinery of harmonic expansions.
Such considerations go beyond the scope of the present work.
\end{remark}

\begin{remark}
In the pure Laplacian case $\nu=\nu_*$ it is known that if $K$ is starshaped with respect to infinity
then \eqref{psi-0} has no positive solutions in the subcritical range $1<p\le p_S$, see \cite[Theorem 2]{Zou}.
This suggests that the nonuniqueness statement of Corollary \ref{cor-0}
can not be extended beyond the supercritical range of exponents.
\end{remark}

\section*{Acknowledgements}

The authors are grateful to Marie-Fran\c{c}oise Bidaut-V\'eron for stimulating discussions.

\end{document}